\documentclass[smallextended,numbook,runningheads]{svjour31}     % onecolumn (second format)
\smartqed  % flush right qed marks, e.g. at end of proof
\usepackage{graphicx}
\usepackage{mathptmx}      % use Times fonts if available on your TeX system

\usepackage{latexsym,bm,amssymb,amsfonts,amsbsy,amsmath,graphics,color}

\newtheorem{theo}{Theorem}
\newtheorem{lem}{Lemma}

\newtheorem{rem}{Remark}

\newtheorem{defn}{Definition}
\def\RR{\mathbb R}
\def\CC{\mathbb C}
\def\pmatrix{ \left( \begin{array} }
\def\endpmatrix{ \end{array} \right) }

\def\cc{\gamma}

\def\bfgam{\bm{\gamma}}

\def\bfpsi{\bm{\psi}}

\def\bfu{{\bf u}}
\def\bfy{{\bf y}}

\def\bfo{{\bf 0}}

\def\no{\noindent}
\def\diag{{\rm diag}}
\def\proof{\underline{Proof}\quad}
\def\QED{~\mbox{$\Box$}}

\def\phi{\varphi}

\def\P{{\cal P}}

\def\bfu{{\bf u}}
\def\bfv{{\bf v}}

\def\I{{\cal I}}
\def\P{{\cal P}}

\def\sigmd{{\dot\sigma}}
\def\O{\Omega}
\def\lam{\lambda}
\def\dd{\mathrm{d}}
\def\bfdel{\boldsymbol{\delta}}

\journalname{~}
\begin{document}

\title{Isospectral Property of Hamiltonian Boundary
Value Methods (HBVMs) and their blended implementation\thanks{Work
developed within the project ``Numerical methods and software for
differential equations''.}} %%\subtitle{Hamiltonian BVMs}

\titlerunning{Blended HBVMs}        % if too long for running head

\author{L.\,Brugnano \and
F.\,Iavernaro \and
 D.\,Trigiante }

\authorrunning{B.I.T.} % if too long for running head

\institute{Luigi Brugnano \at Dipartimento di Matematica,
Universit\`a di Firenze, Viale Morgagni 67/A, 50134 Firenze
(Italy).\\ \email{luigi.brugnano@unifi.it}
\and
Felice Iavernaro \at Dipartimento di Matematica, Universit\`a di Bari,
Via Orabona  4,  70125 Bari (Italy).\\ \email{
felix@dm.uniba.it}
\and
Donato Trigiante \at Dipartimento di Energetica, Universit\`a di Firenze,
Via Lombroso 6/17, 50134 Firenze (Italy).\\ \email{
trigiant@unifi.it} }

\date{February 6, 2010.}% / Accepted: date}

\maketitle

\begin{abstract}
One main issue, when numerically integrating autonomous
Hamiltonian systems, is the long-term conservation of some of its
invariants, among which the Hamiltonian function itself. Recently,
a new class of methods, named {\em Hamiltonian Boundary Value
Methods (HBVMs)} has been introduced and analysed \cite{BIT},
which are able to exactly preserve polynomial Hamiltonians of
arbitrarily high degree. We here study a further property of such
methods, namely that of having, when cast as Runge-Kutta methods,
a matrix of the Butcher tableau with the same spectrum (apart the
zero eigenvalues) as that of the corresponding Gauss-Legendre
method, independently of the considered abscissae. Consequently,
HBVMs are always perfectly $A$-stable methods. Moreover, this
allows their efficient {\em blended} implementation, for solving
the generated discrete problems.

\end{abstract}

\section{Introduction}\label{intro}

Hamiltonian problems are of great interest in many fields of
application, ranging from the macro-scale of celestial mechanics,
to the micro-scale of molecular dynamics. They have been deeply
studied, from the point of view of the mathematical analysis,
since two centuries. Their numerical solution is a more recent
field of investigation, which has led to define symplectic
methods, i.e., the simplecticity of the discrete map, considering
that, for the continuous flow, simplecticity implies the
conservation of $H(y)$. However, the conservation of the
Hamiltonian and simplecticity of the flow cannot be satisfied at
the same time unless the integrator produces the exact solution
(see \cite[page\,379]{HLW}). More recently, the conservation of
energy has been approached by means of the definition of the {\em
discrete line integral}, in a series of papers
\cite{IP1,IP2,IT1,IT2,IT3}, leading to the definition of {\em
Hamiltonian Boundary Value Methods (HBVMs)}
\cite{BIS,BIT2,BIT,BIT1}, which are a class of methods able to
preserve, for the discrete solution, polynomial Hamiltonians of
arbitrarily high degree (and then, a {\em practical} conservation
of any sufficiently differentiable Hamiltonian). In more details,
in \cite{BIT}, HBVMs based on Lobatto nodes have been analysed,
whereas in \cite{BIT1} HBVMs based on Gauss-Legendre abscissae
have been considered. In the last reference, it has been actually
shown that both formulae are essentially equivalent to each other,
in the sense that they share the same order (twice the number of
{\em fundamental stages}) and stability properties, and both
methods provide the very same numerical solution, when the number
of the so called {\em silent stages} tends to
infitiny.\footnote{Actually, they both provide the same numerical
solution also when the Hamiltonian is a polynomial and the number
of {\em silent stages} is high enough to ensure the conservation
property of the Hamiltonian function itself (see
(\ref{discr_lin})).} In this paper this conclusion if further
supported, since we prove that all such methods, when cast as
Runge-Kutta methods, have the corresponding matrix of the tableau,
whose nonzero eigenvalues coincide with those of the corresponding
Gauss-Legendre formula (isospectral property of HBVMs).

This property can be used to define an efficient iteration for
solving the discrete problems generated by the methods, via their
{\em blended implementation}. Indeed, after posing HBVMs in block BVM form, they
can be recast in the framework of {\em blended implicit methods}, which
have been studied in a series of papers
\cite{B00,BM02,BM04,BM07,BM08,BM09,BM09a,BMM06,BT2} (see also
C.\,Magherini's PhD Thesis \cite{M04}). The latter methods have been
successfully implemented in the two computational codes {\tt BiM} and {\tt
BiMD} \cite{BIMcode}; the latter code is also included in the
current release of the ``Test Set for IVP Solvers''
\cite{TestSet}.

With this premise, the structure of the paper is the following: in
Section~\ref{hbvms} the basic facts about HBVMs are recalled; in
Section~\ref{iso} we state the main result of this paper,
concerning the isospectral property; in Section~\ref{blend}
the discrete problem to be actually solved is defined; in
Section~\ref{blend1} it is shown that a corresponding {\em blended
iteration} can be devised for its efficient solution, which can be
tuned by choosing a free parameter; in Section~\ref{linear} the
optimal choice of the free parameter is done, on the basis of the
isospectral property of HBVM$(k,s)$, by using a linear analysis of
convergence; finally, in Section~\ref{fine} a few concluding
remarks are given.

\section{Hamiltonian Boundary Value Methods}\label{hbvms}

The arguments in this section are worked out starting from the
arguments used in \cite{BIT,BIT1} to introduce and analyse HBVMs. We
consider canonical Hamiltonian problems in the form
\begin{equation}
\label{Hamilton} \dot y = J\nabla H(y),  \qquad y(t_0) =
y_0\in\RR^{2m},
\end{equation}

\no where $J$ is a skew-symmetric constant matrix, and the
Hamiltonian $H(y)$ is assumed to be sufficiently differentiable.
The key formula which HBVMs rely on, is the {\em line integral}
and the related property of conservative vector fields:
\begin{equation}\label{Hy}
H(y_1) - H(y_0) = h\int_0^1 \sigmd(t_0+\tau h)^T\nabla
H(\sigma(t_0+\tau h))\dd\tau,
\end{equation}

\no for any $y_1 \in \RR^{2m}$, where $\sigma$ is any smooth
function such that
\begin{equation}
\label{sigma}\sigma(t_0) = y_0, \qquad\sigma(t_0+h) = y_1.
\end{equation}

\no Here we consider the case where $\sigma(t)$ is a polynomial of
degree $s$, yielding an approximation to the true solution $y(t)$
in the time interval $[t_0,t_0+h]$. The numerical approximation
for the subsequent time-step, $y_1$, is then defined by
(\ref{sigma}).

After introducing a set of $s$ distinct abscissae
\begin{equation}\label{ci}0<c_{1},\ldots ,c_{s}\le1,\end{equation}

\no we set
\begin{equation}\label{Yi}Y_i=\sigma(t_0+c_i h), \qquad
i=1,\dots,s,\end{equation}

\no so that $\sigma(t)$ may be thought of as an interpolation
polynomial, interpolating the {\em fundamental stages} $Y_i$,
$i=1,\dots,s$, at the abscissae (\ref{ci}). We observe that, due
to (\ref{sigma}), $\sigma(t)$ also interpolates the initial
condition $y_0$.

\begin{rem}\label{c0} Sometimes, the interpolation at $t_0$ is
explicitly  required. In such a case, the extra abscissa $c_0=0$ is
formally added to (\ref{ci}). This is the case, for example, of a
Lobatto distribution of the abscissae \cite{BIT}.\end{rem}

Let us consider the following expansions of $\dot \sigma(t)$ and
$\sigma(t)$ for $t\in [t_0,t_0+h]$:
\begin{equation}
\label{expan} \dot \sigma(t_0+\tau h) = \sum_{j=1}^{s} \gamma_j
P_j(\tau), \qquad \sigma(t_0+\tau h) = y_0 + h\sum_{j=1}^{s}
\gamma_j \int_{0}^\tau P_j(x)\,\dd x,
\end{equation}

\no where $\{P_j(t)\}$ is a suitable basis of the vector space of
polynomials of degree at most $s-1$ and  the  (vector)
coefficients $\{\gamma_j\}$ are to be determined.  We shall consider an {\em
orthonormal basis} of polynomials on the interval $[0,1]$, i.e.:
\begin{equation}\label{orto}\int_0^1 P_i(t)P_j(t)\dd t = \delta_{ij}, \qquad
i,j=1,\dots,s,\end{equation}

\no where $\delta_{ij}$ is the Kronecker symbol, and $P_i(t)$ has
degree $i-1$. Such a basis can be readily obtained as
$$P_i(t) = \sqrt{2i-1}\,\hat P_{i-1}(t), \qquad
i=1,\dots,s,$$

\no with $\hat P_{i-1}(t)$ the shifted Legendre polynomial, of
degree $i-1$, on the interval $[0,1]$.

\begin{rem}\label{recur}
From the properties of shifted Legendre polynomials (see, e.g.,
\cite{AS} or the Appendix in \cite{BIT}), one readily obtains that
the polynomials $\{P_j(t)\}$ satisfy the three-terms recurrence
relation:
\begin{eqnarray*}
P_1(t)&\equiv& 1, \qquad P_2(t) = \sqrt{3}(2t-1),\\
P_{j+2}(t) &=& (2t-1)\frac{2j+1}{j+1} \sqrt{\frac{2j+3}{2j+1}}
P_{j+1}(t) -\frac{j}{j+1}\sqrt{\frac{2j+3}{2j-1}} P_j(t), \quad
j\ge1.
\end{eqnarray*}
\end{rem}

We shall also assume that $H(y)$ is a polynomial, which implies
that the integrand in \eqref{Hy} is also a polynomial so that the
line integral can be exactly computed by means of a suitable
quadrature formula. It is easy to observe that in general, due to
the high degree of the integrand function,  such quadrature
formula cannot be solely based upon the available abscissae
$\{c_i\}$: one needs to introduce an additional set of abscissae
$\{\hat c_1, \dots,\hat c_r\}$, distinct from the nodes $\{c_i\}$,
in order to make the quadrature formula exact:
\begin{eqnarray} \label{discr_lin}
\displaystyle \lefteqn{\int_0^1 \sigmd(t_0+\tau h)^T\nabla
H(\sigma(t_0+\tau h))\mathrm{d}\tau   =}\\ && \sum_{i=1}^s \beta_i
\sigmd(t_0+c_i h)^T\nabla H(\sigma(t_0+c_i h)) + \sum_{i=1}^r \hat
\beta_i \sigmd(t_0+\hat c_i h)^T\nabla H(\sigma(t_0+\hat c_i h)),
\nonumber
\end{eqnarray}

\no where $\beta_i$, $i=1,\dots,s$, and $\hat \beta_i$,
$i=1,\dots,r$, denote the weights of the quadrature formula
corresponding to the abscissae $\{c_i\}\cup\{\hat c_i\}$, i.e.,
\begin{eqnarray}\nonumber
\beta_i &=& \int_0^1\left(\prod_{ j=1,j\ne i}^s
\frac{t-c_j}{c_i-c_j}\right)\left(\prod_{j=1}^r
\frac{t-\hat c_j}{c_i-\hat c_j}\right)\mathrm{d}t, \qquad i = 1,\dots,s,\\
\label{betai}\\ \nonumber \hat\beta_i &=& \int_0^1\left(\prod_{
j=1}^s \frac{t-c_j}{\hat c_i-c_j}\right)\left(\prod_{ j=1,j\ne
i}^r \frac{t-\hat c_j}{\hat c_i-\hat c_j}\right)\mathrm{d}t,
\qquad i = 1,\dots,r.
\end{eqnarray}

\begin{rem}\label{c01}
In the case considered in the previous Remark~\ref{c0}, i.e. when
$c_0=0$ is formally added to the abscissae (\ref{ci}), the first product in each
formula in (\ref{betai}) ranges from $j=0$ to $s$. Moreover, also the range
of $\{\beta_i\}$ becomes $i=0,1,\dots,s$. However, for sake of
simplicity, we shall not consider this case further.
\end{rem}

According to \cite{IT2}, the right-hand side of \eqref{discr_lin}
is called \textit{discrete line integral}, while the vectors
\begin{equation}\label{hYi}
\hat Y_i \equiv \sigma(t_0+\hat c_i h), \qquad i=1,\dots,r,
\end{equation}

\no are called \textit{silent stages}: they just serve to
increase, as much as one likes, the degree of precision of the
quadrature formula, but they are not to be regarded as unknowns
since, from \eqref{expan} and (\ref{hYi}), they can be expressed in terms of
linear combinations of the fundamental stages (\ref{Yi}).

\begin{defn}\label{defhbvmks}
The method defined by substituting the quantities in \eqref{expan}
into the right-hand side of \eqref{discr_lin}, and by choosing the
unknown coefficients $\{\gamma_j\}$ in order that the resulting
expression vanishes, is called {\em Hamiltonian Boundary Value
Method with $k$ steps and degree $s$}, in short {\em HBVM($k$,$s$)},
where $k=s+r$ \, \cite{BIT}.\end{defn}

In the sequel, we shall see that HBVMs may be expressed through different,
though equivalent, formulations: some of them can be directly implemented in a
computer program, the others being of more theoretical interest.

Because of the equality \eqref{discr_lin}, we can apply the
procedure directly to the original line integral appearing in the
left-hand side. With this premise, by considering the first expansion in
\eqref{expan},  the conservation property  reads
$$\sum_{j=1}^{s} \gamma_j^T \int_0^1  P_j(\tau) \nabla
H(\sigma(t_0+\tau h))\dd\tau=0,$$

\no which, as is easily checked,  is certainly satisfied if we
impose the following set of orthogonality conditions:
\begin{equation}
\label{orth} \gamma_j = \int_0^1  P_j(\tau) J \nabla
H(\sigma(t_0+\tau h))\dd\tau, \qquad j=1,\dots,s.
\end{equation}

\no Then, from the second relation of \eqref{expan} we obtain, by
introducing the operator
\begin{eqnarray}\label{Lf}\lefteqn{L(f;h)\sigma(t_0+ch) =}\\
\nonumber && \sigma(t_0)+h\sum_{j=1}^s \int_0^c P_j(x) \dd x \,
\int_0^1 P_j(\tau)f(\sigma(t_0+\tau h))\dd\tau,\qquad
c\in[0,1],\end{eqnarray}

\no that $\sigma$ is the eigenfunction of $L(J\nabla H;h)$
relative to the eigenvalue $\lambda=1$:
\begin{equation}\label{L}\sigma = L(J\nabla H;h)\sigma.\end{equation}

\begin{defn} Equation (\ref{L}) is the {\em Master Functional
Equation} (MFE) defining $\sigma$ ~\cite{BIT1}.\end{defn}

\begin{rem}\label{MFE}
Some further details are in order to better elucidate the role of
the MFE in devising our methods. First of all we observe that, by
definition, the MFE intrinsically brings, with its polynomial
solutions $\sigma(t_0+ch)$, the conservation property of the
Hamiltonian function: indeed \eqref{L}  is equivalent to
\eqref{discr_lin} under the choice \eqref{orth}.

This means that, when searching for its solutions, one should always
take care of the precise dimension of the polynomial vector space,
say $\nu$,  $H(y)$ is intended to belong to: the higher is $\nu$,
the higher must be the number of silent stages (and hence the number
of steps $k$) to guarantee that \eqref{discr_lin} be satisfied. This
explains the way the solutions of the MFE depends on $k$.

It is also clear that, assuming the same kind of distribution for
all the $k$ the nodes (see later), \eqref{discr_lin} will be
satisfied starting from a suitable number of steps $k\equiv k_\nu$
on. This implies that, for all $k\ge k_\nu$, HBVM$(k,s)$ will
define the {\em same} polynomial $\sigma$ of degree $s$, such that
$H(\sigma(t_0+h))=H(\sigma(t_0))$.
\end{rem}

To practically compute $\sigma$, we set (see (\ref{Yi}) and
(\ref{expan}))
\begin{equation}
\label{y_i} Y_i=  \sigma(t_0+c_i h) = y_0+ h\sum_{j=1}^{s} a_{ij}
\gamma_j, \qquad i=1,\dots,s,
\end{equation}

\no where
$$a_{ij}=\int_{0}^{c_i} P_j(x) \mathrm{d}x, \qquad
i,j=1,\dots,s.$$%

\no Inserting \eqref{orth} into \eqref{y_i} yields the final
formulae which define the HBVMs class based upon the orthonormal
basis $\{P_j\}$:
\begin{equation}
\label{hbvm_int} Y_i=y_0+h  \sum_{j=1}^s a_{ij}\int_0^1
P_j(\tau)  J \nabla H(\sigma(t_0+\tau h))\,\dd\tau, \qquad
i=1,\dots,s.
\end{equation}

We recall once again that we are working under the assumption
\eqref{discr_lin}, namely that the Hamiltonian is a polynomial and
that we are considering a sufficient number of additional
abscissae $\hat c_i$ such that the line integral and its discrete
counterpart do coincide. This implies that  we can replace the
integrals appearing in \eqref{hbvm_int} by sums representing the
associated quadrature formulae  introduced in \eqref{discr_lin},
without introducing any discretization error.

This leads back to express the HBVM$(k,s)$ method in terms of the
fundamental stages $\{Y_i\}$ and the silent stages $\{\hat Y_i\}$
(see (\ref{hYi})). By using the notation
\begin{equation}\label{fy}
f(y) = J \nabla H(y),
\end{equation}

\no we obtain
\begin{equation}
\label{hbvm_sys} Y_i = y_0+h\sum_{j=1}^s a_{ij}\left( \sum_{l=1}^s
\beta_l P_j(c_l)f(Y_l) + \sum_{l=1}^r\hat \beta_l P_j(\hat c_l)
f(\widehat Y_l) \right),\quad i=1,\dots,s.
\end{equation}

\no We again stress that the silent stages $\hat Y_l$ may be
removed from \eqref{hbvm_sys} by observing that, for example,
$$
\hat Y_l = \sum_{i=1}^s \ell(t_0+\hat c_ih) Y_i, \qquad
l=1,\dots,r,
$$

\no where $\ell(t)$ are the cardinal Lagrange polynomials defined
on the nodes $t_0+c_ih$, $i=1,\dots,s$.

From the above discussion it is clear that formulae \eqref{hbvm_int}
also make sense in the non-polynomial case. In fact, supposing to
choose the abscissae $\{\hat c_i\}$ so that the sums in
(\ref{hbvm_sys}) converge to an integral as $r\equiv
k-s\rightarrow\infty$, the resulting formula is again
\eqref{hbvm_int}.

\begin{defn}\label{infh} Formula (\ref{hbvm_int}) is named {\em
$\infty$-HBVM of degree $s$} or {\em HBVM$(\infty,s)$} \, \cite{BIT1}.
\end{defn}

This implies that HBVMs may be as well applied in the non-polynomial case since,
in finite precision arithmetic, HBVMs are undistinguishable from their limit
formulae \eqref{hbvm_int}, when a sufficient number of silent stages is
introduced. The aspect of having a {\em practical} exact integral,
for $k$ large enough, was already stressed in
\cite{BIS,BIT,BIT1,IP1,IT2}.

\subsection{Runge-Kutta formulation of HBVMs} On the other hand,
we emphasize that, in the non-polynomial case, \eqref{hbvm_int}
becomes an {\em operative method} only after that a suitable
strategy to approximate the integrals appearing in it is taken
into account. In the present case, if one discretizes the {\em
Master Functional Equation} (\ref{Lf})--(\ref{L}), HBVM$(k,s)$ are
then obtained, essentially by extending the discrete problem
(\ref{hbvm_sys}) also to the silent stages (\ref{hYi}). In order
to simplify the exposition, we shall use (\ref{fy}) and introduce
the following notation:
\begin{eqnarray*}
\{t_i\} = \{c_i\} \cup \{\hat{c}_i\}, &&
\{\omega_i\}=\{\beta_i\}\cup\{\hat\beta_i\},\\[2mm]
y_i = \sigma(t_0+t_ih), && f_i =
f(\sigma(t_0+t_ih)), \qquad i=1,\dots,k.
\end{eqnarray*}

\no The discrete problem defining the HBVM$(k,s)$ then becomes,
\begin{equation}\label{hbvmks}
y_i = y_0 + h\sum_{j=1}^s \int_0^{t_i} P_j(x)\dd x \sum_{\ell=1}^k
\omega_\ell P_j(t_\ell)f_\ell, \qquad i=1,\dots,k.
\end{equation}

By introducing the vectors $$\bfy = (y_1^T,\dots,y_k^T)^T, \qquad
e=(1,\dots,1)^T\in\RR^k,$$ and the matrices
\begin{equation}\label{OIP}\O=\diag(\omega_1,\dots,\omega_k), \qquad
\I_s,~\P_s\in\RR^{k\times s},\end{equation} whose $(i,j)$th entry
are given by
\begin{equation}\label{IDPO}
(\I_s)_{ij} = \int_0^{t_i} P_j(x)\mathrm{d}x, \qquad
(\P_s)_{ij}=P_j(t_i), \end{equation}

\no we can cast the set of equations (\ref{hbvmks}) in vector form
as $$\bfy = e\otimes y_0 + h(\I_s
\P_s^T\O)\otimes I_{2m}\, f(\bfy),$$

\no with an obvious meaning of $f(\bfy)$. Consequently, the method
can be seen as a Runge-Kutta method with the following Butcher
tableau:
\begin{equation}\label{rk}
\begin{array}{c|c}\begin{array}{c} t_1\\ \vdots\\ t_k\end{array} & \I_s \P_s^T\O\\
 \hline                    &\omega_1\, \dots~ \omega_k
                    \end{array}\end{equation}

\begin{rem}\label{ascisse} We observe that, because of the use of an
orthonormal basis, the role of the abscissae $\{c_i\}$ and of the
silent abscissae $\{\hat c_i\}$ is interchangeable, within the set
$\{t_i\}$. This is due to the fact that all the matrices $\I_s$,
$\P_s$, and $\O$ depend on all the abscissae $\{t_i\}$, and not on
a subset of them, and they are invariant with respect to the choice of the
fundamental abscissae $\{c_i\}$.
\end{rem}

Hereafter, we shall consider a Gauss distribution of the abscissae
$\{t_1,\dots,t_k\}$, so that the resulting  HBVM$(k,s)$ method
\cite{BIT1}:

\begin{itemize}

\item has order $2s$ for all $k\ge s$;

\item is symmetric and perfectly $A$-stable (i.e., its
stability region coincides with the left-half complex plane,
$\CC^-$ \cite{BT});

\item reduces to the Gauss-Legendre method of order $2s$, when
$k=s$;

\item exactly preserves polynomial Hamiltonian functions of degree $\nu$,
provided that $$k\ge \frac{\nu s}2.$$

\end{itemize}

\section{The Isospectral Property}\label{iso}

We are now going to prove a further additional result, related to
the matrix appearing in the Butcher tableau (\ref{rk}),
corresponding to HBVM$(k,s)$, i.e., the matrix
\begin{equation}\label{AMAT}A = \I_s \P_s^T\O\in\RR^{k\times
k}, \qquad k\ge s,\end{equation}

\no whose rank is $s$. Consequently it has a $(k-s)$-fold zero
eigenvalue. In this section, we are going to discuss the location
of the remaining $s$ eigenvalues of that matrix.

Before that, we state the following preliminary result, whose
proof can be found in \cite[page\,79]{HW}.

\begin{lem}\label{gauss} The eigenvalues of the matrix
\begin{equation}\label{Xs}
X_s = \pmatrix{cccc}
\frac{1}2 & -\xi_1 &&\\
\xi_1     &0      &\ddots&\\
          &\ddots &\ddots    &-\xi_{s-1}\\
          &       &\xi_{s-1} &0\\
\endpmatrix, \end{equation} with
\begin{equation}\label{xij}\xi_j=\frac{1}{2\sqrt{(2j+1)(2j-1)}}, \qquad
j\ge1,\end{equation} coincide with those of the matrix in the
Butcher tableau of the Gauss-Legendre method of order
$2s$.\end{lem}

\medskip
We also need the following preliminary result, whose proof derives
from the properties of shifted-Legendre polynomials (see, e.g.,
\cite{AS} or the Appendix in \cite{BIT}).

\begin{lem}\label{intleg} With reference to the matrices in
(\ref{OIP})--(\ref{IDPO}), one has
$$\I_s = \P_{s+1}\hat{X}_s,$$

\no where
$$\hat{X}_s = \pmatrix{cccc}
\frac{1}2 & -\xi_1 &&\\
\xi_1     &0      &\ddots&\\
          &\ddots &\ddots    &-\xi_{s-1}\\
          &       &\xi_{s-1} &0\\
\hline &&&\xi_s\endpmatrix,$$

\no with the $\xi_j$ defined by (\ref{xij}). \end{lem}

\medskip
The following result then holds true.

\begin{theo}[Isospectral Property of HBVMs]\label{mainres}
For all $k\ge s$ and for any choice of the abscissae $\{t_i\}$
such that the quadrature defined by the weights $\{\omega_i\}$ is exact for
polynomials of degree $2s-1$, the nonzero eigenvalues of the
matrix $A$ in (\ref{AMAT}) coincide with those of the matrix of
the Gauss-Legendre method of order $2s$.
\end{theo}
\begin{proof}
For $k=s$, the abscissae $\{t_i\}$ have to be the $s$
Gauss-Legendre nodes, so that HBVM$(s,s)$ reduces to the Gauss
Legendre method of order $2s$, as outlined at the end of Section~\ref{hbvms}.

When $k>s$, from the orthonormality of the basis, see
(\ref{orto}), and considering that the quadrature with weights
$\{\omega_i\}$ is exact for polynomials of degree (at least)
$2s-1$, one easily obtains that (see (\ref{OIP})--(\ref{IDPO}))
$$\P_s^T\O\P_{s+1} = \left( I_s ~ \bfo\right),$$

\no since, for all ~$i=1,\dots,s$,~ and ~$j=1,\dots,s+1$:
$$\left(\P_s^T\O\P_{s+1}\right)_{ij} = \sum_{\ell=1}^k \omega_\ell
P_i(t_\ell)P_j(t_\ell)=\int_0^1 P_i(t)P_j(t)\dd t = \delta_{ij}.$$

\no By taking into account the result of Lemma~\ref{intleg}, one
then obtains:
\begin{eqnarray*}
A\P_{s+1} &=& \I_s \P_s^T\O\P_{s+1} = \I_s \left(I_s~\bfo\right)
=\P_{s+1}
\hat{X}_s \left(I_s~\bfo\right) = \P_{s+1}\left(\hat{X}_s~\bfo\right)\\
 &\equiv& \P_{s+1}
\pmatrix{cccc|c}
\frac{1}2 & -\xi_1 && &0\\
\xi_1     &0      &\ddots& &\vdots\\
          &\ddots &\ddots    &-\xi_{s-1}&\vdots\\
          &       &\xi_{s-1} &0&0\\
\hline &&&\xi_s&0\endpmatrix\end{eqnarray*}

\no with the $\{\xi_j\}$ defined according to (\ref{xij}).
Consequently, one obtains that the columns of $\P_{s+1}$
constitute a basis of an invariant (right) subspace of matrix $A$,
so that the eigenvalues of $(\hat X_s~\bfo)$ are eigenvalues of
$A$. In more detail, the eigenvalues of $(\hat X_s~\bfo)$ are
those of $X_s$ (see (\ref{Xs})) and the zero eigenvalue. Then,
also in this case, the nonzero eigenvalues of $A$ coincide with
those of $X_s$, i.e., with the eigenvalues of the matrix defining
the Gauss-Legendre method of order $2s$.\QED\end{proof}

\section{Solving the discrete problem}\label{blend}

We shall now consider some computational aspects concerning
HBVM$(k,s)$. In more details, we now show how its cost depends
essentially on $s$, rather than on $k$, in the sense that the
nonlinear system to be solved, for obtaining the discrete solution,
has (block) dimension $s$. This has been already shown in
\cite{BIT}, but here we derive the same result in a slightly more
compact way, which will allow us to easily introduce {\em blended
HBVMs} in the next section.

In order to simplify the notation, we shall fix the fundamental
stages at $t_1,\dots,t_s$, since we have already seen that, due to
the use of an orthonormal basis , they could be in principle chosen
arbitrarily, among the abscissae $\{t_i\}$. With this premise, we
have, from (\ref{hbvm_int}),
\begin{equation}\label{ys}
y_i = y_0 + h\sum_{j=1}^s a_{ij}  \sum_{\ell = 1}^k\omega_\ell
P_j(t_\ell)f_\ell, \qquad i = 1,\dots,s.\end{equation}

This equation is now coupled with that defining the silent stages,
i.e., from (\ref{expan}) and (\ref{hYi}),
\begin{equation}\label{hy}
y_i = y_0 + \sum_{j=1}^s \cc_j \int_0^{t_i}P_j(t) \dd t, \qquad i
= s+1,\dots,k.
\end{equation}

Let us now partition the matrices $\I_s,\P_s\in\RR^{k\times s}$ in
(\ref{OIP})--(\ref{IDPO}) into $\I_{s1},\P_{s1}\in\RR^{s\times s}$ and
$\I_{s2},\P_{s2}\in\RR^{k-s\times s}$, containing the entries
defined by the fundamental abscissae and the silent abscissae,
respectively. Similarly, we partition the vector $\bfy$ into
$\bfy_1$, containing the fundamental stages, and $\bfy_2$
containing the silent stages and, accordingly, let
$\O_1\in\RR^{s\times s}$ and $\O_2\in\RR^{k-s\times k-s}$ be the
diagonal matrices containing the corresponding entries in matrix
$\O$. Finally, let us define the vectors $\bfgam =
(\cc_1^T,\dots,\cc_s^T)^T$, $e=(1,\dots,1)^T\in\RR^s$, and $u =
(1,\dots,1)^T\in\RR^{k-s}$.

\no Consequently, we can rewrite (\ref{ys}) and (\ref{hy}), as
\begin{eqnarray}\label{y1}
\bfy_1 &=& e\otimes y_0 + h\I_{s1} \left( \P_{s1}^T~
\P_{s2}^T\right) \pmatrix{cc}\O_1 \\ &\O_2\endpmatrix\otimes
I_{2m}\pmatrix{c} f(\bfy_1)\\
f(\bfy_2)\endpmatrix,\\ \bfy_2 &=& u\otimes y_0 +h \I_{s2}\otimes
I_{2m} \bfgam, \label{y2}\end{eqnarray}

\no respectively. The vector $\bfgam$ can be easily retrieved from the
identity (\ref{y_i}), which in vector form reads
$$\bfy_1 = e\otimes y_0 + h\I_{s1}\otimes I_{2m} \bfgam,$$

\no thus giving \begin{eqnarray}\nonumber \bfy_2 &=&
\left(u-\I_{s2}\I_{s1}^{-1}e\right)\otimes y_0
+\I_{s2}\I_{s1}^{-1}\otimes I_{2m} \bfy_1\\ &\equiv& \hat u\otimes
y_0 +A_1\otimes I_{2m}\bfy_1,\label{A1}\end{eqnarray}

\no in place of (\ref{y2}), where, evidently,
$A_1\in\RR^{k-s\times s}$. By setting
\begin{equation}\label{B1B2}
B_1 = \I_{s1} \P_{s1}^T\O_1 \in\RR^{s\times s}, \qquad B_2 = \I_{s1}
\P_{s2}^T \O_2\in\RR^{s\times k-s},\end{equation}

\no substitution of (\ref{A1}) into (\ref{y1}) then provides, at
last, the system of (block) size $s$ to be actually solved:
\begin{eqnarray}\label{onlys}
F(\bfy_1) &\equiv& \bfy_1 - e\otimes y_0 - h\left[ B_1 \otimes
I_{2m} f(\bfy_1) + \right.\\ &&\left. B_2 \otimes I_{2m}
f\left(\hat u\otimes y_0 +A_1\otimes I_{2m} \bfy_1\right)\right] =
\bf0.\nonumber\end{eqnarray}

By using the simplified Newton method for solving (\ref{onlys}),
and setting
\begin{equation}\label{C} C=B_1+B_2A_1 \in\RR^{s\times s},\end{equation}

\no one obtains the iteration:
\begin{eqnarray}\label{Newt}
 \left( I_s\otimes I_{2m} - hC\otimes J_0 \right) \bfdel^{(n)} &=&
-F(\bfy_1^{(n)}) \equiv \bfpsi_1^{(n)},\\ \bfy_1^{(n+1)} &=&
\bfy_1^{(n)} + \bfdel^{(n)}, \qquad n=0,1,\dots, \nonumber
\end{eqnarray}

\no where $J_0$ is the Jacobian of $f(y)$ evaluated at $y_0$.
Because of the result of Theorem~\ref{mainres}, the following
property of matrix $C$ holds true.

\begin{theo}\label{isoC}
The eigenvalues of matrix $C$ in (\ref{C}) coincide with those of
matrix (\ref{Xs}), i.e., with the eigenvalues of the matrix of the
Butcher array of the Gauss-Legendre method of order $2s$.
\end{theo}

\proof Assuming, as usual for simplicity, that the fundamental
stages are the first $s$ ones, one has that the discrete problem
$$\bfy = \pmatrix{c}e\\ u\endpmatrix \otimes y_0 +h A\otimes
I_{2m} f(\bfy),$$

\no which defines the Runge-Kutta formulation of the method, is
equivalent, by virtue of (\ref{y1}), (\ref{A1}), (\ref{B1B2}), to
\begin{eqnarray*}\lefteqn{\pmatrix{cc} I_s & O_{s\times r}\\ -A_1 & I_r
\endpmatrix\otimes
I_{2m} \pmatrix{c} \bfy_1\\ \bfy_2\endpmatrix =}\\&& \pmatrix{c} e\\
\hat u\endpmatrix\otimes y_0 +h\pmatrix{cc} B_1 & B_2\\
O_{r\times s} & O_{r\times r} \endpmatrix\otimes I_{2m} \pmatrix{c} f(\bfy_1)\\
f(\bfy_2)\endpmatrix,\end{eqnarray*}

\no where, as usual, $r=k-s$.\footnote{As observed in \cite{IP2,IT3}, such
formulation fits the framework of block BVMs.}
Consequently, the eigenvalues of the matrix $A$ defined in (\ref{AMAT})
coincide with those of the matrix pencil
$$\left(~\pmatrix{cc} I_s &O_{s\times r}\\ -A_1 & I_r\endpmatrix,~
\pmatrix{cc} B_1 &B_2\\ O_{r\times s} & O_{r\times
r}\endpmatrix~\right),$$

\no that is
$$\mu\in\sigma(A) ~~\Leftrightarrow~~ \mu\pmatrix{cc} I_s &O_{s\times r}\\
-A_1 & I_r\endpmatrix\pmatrix{c} \bfu\\ \bfv\endpmatrix =\pmatrix{cc} B_1 &B_2\\
O_{r\times s} & O_{r\times r}\endpmatrix\pmatrix{c} \bfu\\
\bfv\endpmatrix,$$

\no for some nonzero vector $(\bfu^T,\bfv^T)^T$. By setting
$\bfu=\bfo$, one obtains the $r$ zero eigenvalues of the pencil.
For the remaining $s$ (nonzero) ones, it must be $\bfv=A_1\bfu$,
so that:
$$\mu \bfu = \left( B_1\bfu + B_2\bfv \right) = \left( B_1\bfu + B_2A_1\bfu
\right) = C\bfu ~~\Leftrightarrow~~ \mu\in\sigma(C).\QED$$

\medskip
\begin{rem}\label{algo}
From the result of Theorem~\ref{isoC}, it follows that the
spectrum of $C$ doesn't depend on the choice of the $s$
fundamental abscissae, within the nodes $\{t_i\}$. On the
contrary, its condition number does: the latter appears to be
minimized when the fundamental abscissae are symmetrically
distributed, and approximately evenly spaced, in the interval
$[0,1]$. As a practical ``{\em rule of thumb}'', the following
algorithm appears to be almost optimal:\footnote{We plan to investigate this aspect
further, in a forthcoming paper.}
\begin{enumerate}
\item let the $k$ abscissae $\{t_i\}$ be chosen according to a
Gauss-Legendre distribution of $k$ nodes;

\item then, let us consider $s$ equidistributed
nodes in $(0,1)$, say $\{\hat t_1, \dots, \hat t_s\}$;

\item select, as the fundamental abscissae, those nodes, among the
$\{t_i\}$, which are the closest ones to the $\{\hat t_j\}$;

\item define matrix $C$ in (\ref{C}) accordingly.
\end{enumerate}

\no Clearly, for the above algorithm to provide a unique solution
(resulting in a symmetric choice of the fundamental abscissae),
the difference $k-s$ has to be even which, however, can be easily
accomplished.
\end{rem}

In order to give evidence of the effectiveness of the above
algorithm, in Figure~\ref{condC} we plot the condition number of
matrix $C=C(k,s)$, for $s=2,\dots,5$, and $k\ge s$. As one can
see, the condition number of $C(k,s)$ turns out to be nicely
bounded, for increasing values of $k$, which makes the
implementation (that we are going to analyze in the next section)
effective also when finite precision arithmetic is used. For
comparison, in Figure~\ref{condC1}  there is the same plot,
obtained by fixing the fundamental abscissae as the first $s$
ones.

\begin{figure}[hp]
\centerline{\includegraphics[width=\textwidth,height=8.5cm]{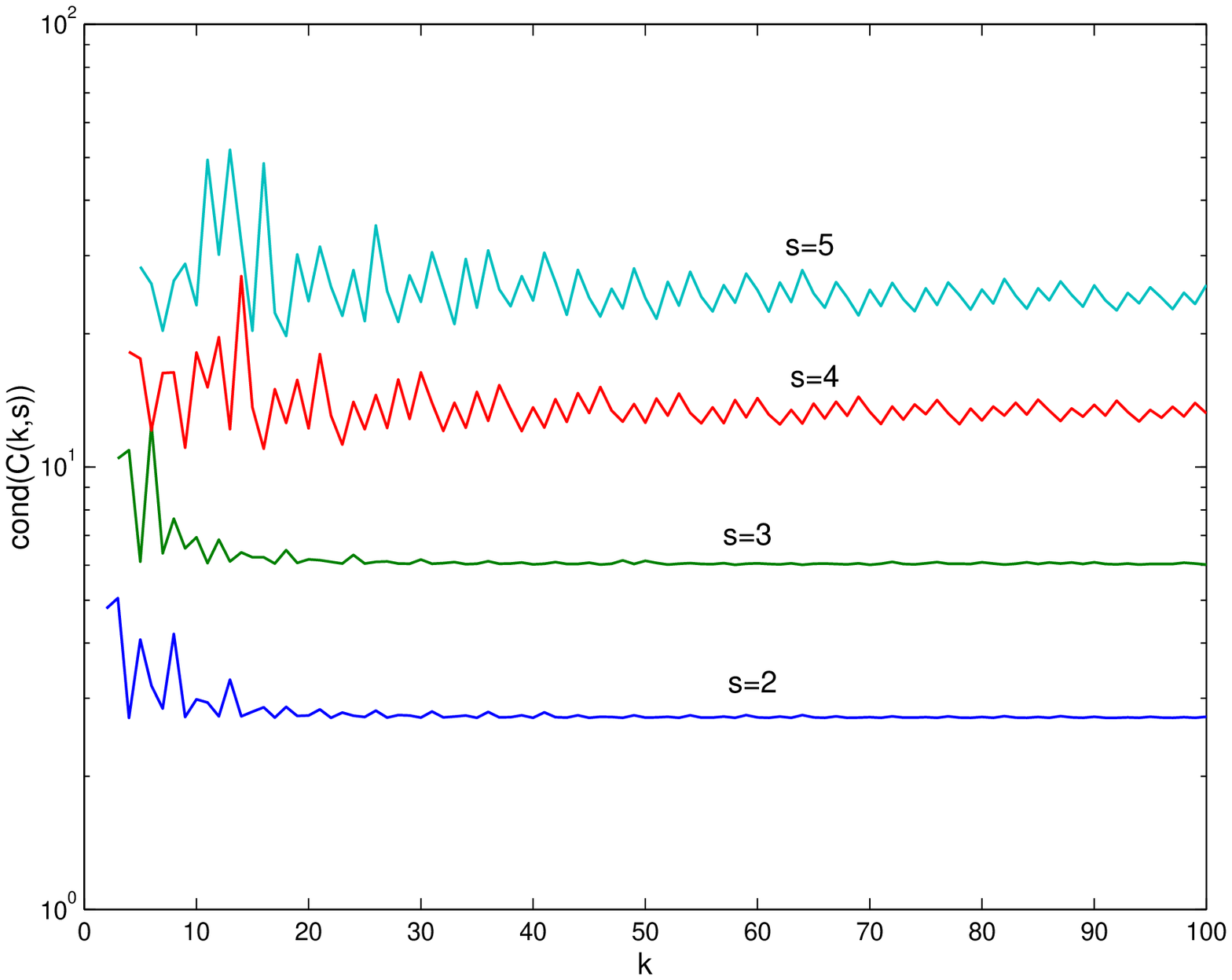}}
\caption{\protect\label{condC} Condition number of the matrix
$C=C(k,s)$, for $s=2,3,4,5$ and $k=s,s+1,\dots,100$, with  the
fundamental abscissae chosen according to the algorithm sketched
in Remark~\ref{algo}.}
\bigskip
\medskip
\centerline{\includegraphics[width=\textwidth,height=8.5cm]{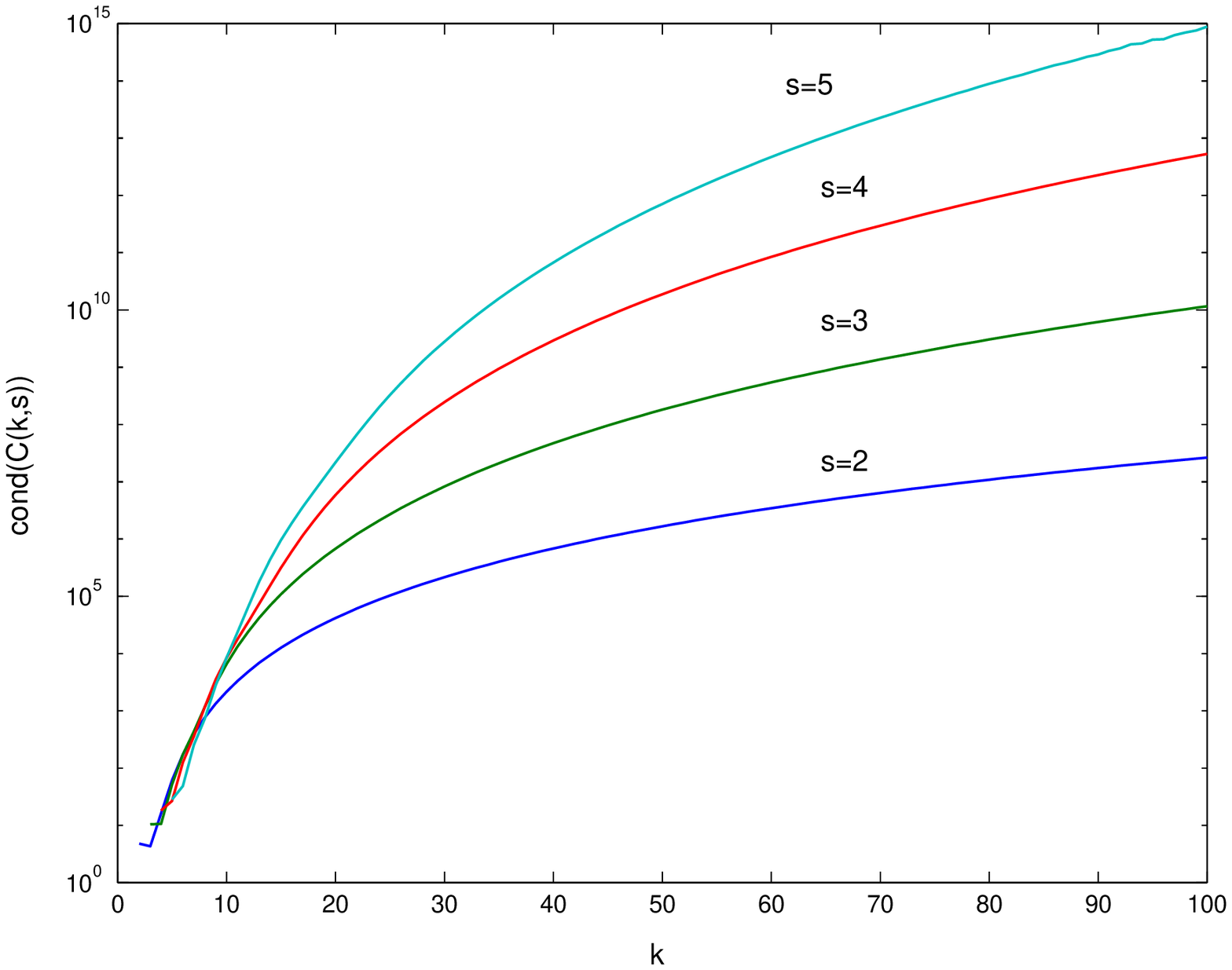}}
\caption{\protect\label{condC1} Condition number of the matrix
$C=C(k,s)$, for $s=2,3,4,5$ and $k=s,s+1,\dots,100$, with the
fundamental  abscissae chosen as the first $s$ ones.}
\end{figure}

\section{Blended HBVMs}\label{blend1}

The solution of problem (\ref{Newt}) is now cast into the framework of {\em
blended implicit methods} \cite{B00,BM02,BM04,BM07,BM08,BMM06,BT2,M04} as below
described. First of all, we observe that, since $C$ is nonsingular, we can
recast problem
(\ref{Newt}) in the {\em equivalent form}
\begin{equation}\label{Newt2}
 \gamma\left( C^{-1}\otimes I_{2m} - hI_s\otimes J_0 \right) \bfdel^{(n)} =
-\gamma C^{-1}\otimes I_{2m}\, F(\bfy_1^{(n)}) \equiv
\bfpsi_2^{(n)},
\end{equation}

\no where $\gamma>0$ is a free parameter to be chosen later. Let
us now introduce the {\em weight (matrix) function}
\begin{equation}\label{teta}
 \theta = I_s\otimes \Phi^{-1}, \qquad \Phi = I_{2m} -h\gamma
J_0\in\RR^{2m\times 2m},
\end{equation}

\no and the {\em blended formulation} of the system to be solved,
\begin{eqnarray}\nonumber
 M\bfdel^{(n)} &\equiv& \left[ \theta \left(I_s\otimes I_{2m}-hC\otimes
J_0\right) + \right.\\ \nonumber &&\left. (I-\theta)\gamma\left(
C^{-1}\otimes I_{2m}- h I_s\otimes J_0 \right)\right] \bfdel^{(n)}
\\ &=& \theta \bfpsi_1^{(n)}+(I-\theta)\bfpsi_2^{(n)}\equiv
\bfpsi^{(n)}.\label{blendNewt}
\end{eqnarray}

\no The latter system has still the same solution as the previous
ones, since it is obtained as the {\em blending}, with weights
$\theta$ and $(I-\theta)$, of the two equivalent forms
(\ref{Newt}) and (\ref{Newt2}). For iteratively solving
(\ref{blendNewt}), we use the corresponding {\em blended
iteration}, formally given by \cite{B00,BM02,BM04,BM07,BM08,BMM06,BT2,M04}:
\begin{equation}\label{blendit}\bfdel^{(n,\ell+1)} = \bfdel^{(n,\ell)}
-\theta\left( M\bfdel^{(n,\ell)}-\bfpsi^{(n)}\right), \qquad
\ell=0,1,\dots.\end{equation}

\begin{rem}\label{nonlin}
A nonlinear variant of the iteration (\ref{blendit}) can be
obtained, by setting
$$\bfy^{(n,\ell+1)} = \bfy^{(n,\ell)}+\bfdel^{(n,\ell)}, \qquad
\bfpsi_1^{(n,\ell)} = -F\left(\bfy_1^{(n,\ell)}\right),$$

\no $\bfpsi_2^{(n,\ell)}$ and $\bfpsi^{(n,\ell)}$ similarly
defined, as:
\begin{equation}\label{blendit1}\bfdel^{(n,\ell+1)} = \bfdel^{(n,\ell)}
-\theta\left( M\bfdel^{(n,\ell)}-\bfpsi^{(n,\ell)}\right), \qquad
\ell=0,1,\dots.\end{equation}\end{rem}

\begin{rem}\label{only1} We emphasize that, for actually performing the
iteration (\ref{teta})--(\ref{blendit}), as well as
(\ref{blendit1}), one has to factor only the matrix $\Phi$ in
(\ref{teta}), which has the same size as that of the continuous
problem, due to the (block) diagonal structure of \,$\theta$.\end{rem}

We end this section by observing that the above iteration
(\ref{blendit}) depends on a free parameter $\gamma$. It will be
chosen in order to optimize the convergence properties of the
iteration,  according to a linear analysis of convergence, which
is sketched in the next section.

\section{Linear analysis of convergence}\label{linear}

The linear analysis of convergence for the iterations
(\ref{blendit}) is carried out by considering the usual scalar
test equation (see, e.g., \cite{BM09} and the references therein),
$$y' = \lam y, \qquad \Re(\lam)<0.$$

\no By setting, as usual $q=h\lam$, the two equivalent
formulations (\ref{Newt}) and (\ref{Newt2}) become, respectively
(omitting, for sake of brevity, the upper index $n$),
$$(I_s-q C) \bfdel = \bfpsi_1, \qquad \gamma( C^{-1} -q I_s) \bfdel =
\bfpsi_2.$$

\no Moreover, \begin{equation}\label{tetaq}\theta \equiv \theta(q) =
(1-\gamma q)^{-1} I_s,\end{equation}

\no and the blended iteration (\ref{blendit}) becomes
\begin{equation}\label{blendq}
\bfdel^{(\ell+1)} = (I_s -\theta(q)M(q))\bfdel^{(\ell)} +
\theta(q)\bfpsi(q),\end{equation}

\no with
\begin{eqnarray}\label{Mq} M(q) &=&
\theta(q)\left(I_s-q C\right) +(I_s-\theta(q))\gamma\left(
C^{-1}-q I_s\right), \\ \bfpsi(q) &=&
\theta(q)\bfpsi_1+(I_s-\theta(q))\bfpsi_2.\nonumber\end{eqnarray}

\no Consequently, the iteration will be convergent if and only if
the spectral radius, say $\rho(q)$, of the iteration matrix,
\begin{equation}\label{Zq}Z(q) =
I_s-\theta(q)M(q),\end{equation}

\no is less than 1. The set
$$\Gamma = \left\{ q\in\CC \,:\, \rho(q)<1 \right\}$$

\no is the {\em region of convergence of the iteration}. The
iteration is said to be (see \cite{BM09} for
details):\begin{itemize}
\item {$A$-convergent}, ~ if $\CC^-\subseteq \Gamma$;

\item {$L$-convergent}, ~ if it is $A$-convergent and,
moreover, ~$\rho(q)\rightarrow 0$,~ as ~$q\rightarrow\infty$.
\end{itemize}
For the iteration (\ref{blendq}) one verifies that (see
(\ref{tetaq}), (\ref{Mq}), and (\ref{Zq}))
\begin{equation}\label{Zq1}
Z(q) = \frac{q}{(1-\gamma q)^2}C^{-1}\left(C-\gamma I_s\right)^2,
\end{equation}

\no which is the null matrix at $q=0$ and at $\infty$.
Consequently, the iteration will be $A$-convergent (and,
therefore, $L$-convergent), provided that {\em maximum
amplification factor},
$$\rho^* \equiv \max_{\Re(q)=0} \rho(q) ~\le 1.$$

\no From (\ref{Zq1}) one has that\,\footnote{Hereafter,
$\sigma(C)$ will denote the spectrum of matrix $C$.}
$$\mu\in\sigma(C)
~\Leftrightarrow~\frac{q(\mu-\gamma)^2}{\mu(1-\gamma
q)^2}\in\sigma(Z(q)).$$

\no By taking into account that
$$\max_{\Re(q)=0}\frac{|q|}{|(1-\gamma q)^2|} =
\frac{1}{2\gamma},$$

\no one then obtains that
$$\rho^* = \max_{\mu\in\sigma(C)}
\frac{|\mu-\gamma|^2}{2\gamma|\mu|}.$$

\no For the Gauss-Legendre methods (and, then, for any matrix $C$
having the same spectrum), it can be shown that \cite{BM02,BMM06}
the choice
\begin{equation}\label{gammaopt} \gamma = |\mu_{\min}|\equiv
\min_{\mu\in\sigma(C)}|\mu|,\end{equation}

\no minimizes $\rho^*$, which turns out to be given by
\begin{equation}\label{rostarmin} \rho^* = 1 -\cos \phi_{\min}
~<1, \qquad \phi_{\min}={\rm Arg}(\mu_{\min}). \end{equation}

In Table~\ref{params}, we list the optimal value of the parameter
$\gamma$, along with the corresponding maximum amplification
factor $\rho^*$, for various values of $s$, which confirm that the
iteration (\ref{blendq}) is $L$-convergent.

\begin{table}
\caption{\protect\label{params} Optimal values (\ref{gammaopt}),
and corresponding maximum amplification factors (\ref{rostarmin}),
for various values of $s$.} \centerline{\begin{tabular}{|r|r|r|}
\hline
$s$ & $\gamma$ & $\rho^*$\\
\hline
2 &0.2887 &0.1340\\
3 &0.1967 &0.2765\\
4 &0.1475 &0.3793\\
5 &0.1173 &0.4544\\
6 &0.0971 &0.5114\\
7 &0.0827 &0.5561\\
8 &0.0718 &0.5921\\
9 &0.0635 &0.6218\\
10 &0.0568 &0.6467\\
\hline
\end{tabular}}
\end{table}
\begin{rem}
We then conclude that the {\em blended iteration} (\ref{blendit})
turns out to be $L$-conver\-gent for HBVM$(k,s)$ methods, for
all $s\ge1$ and $k\ge s$.
\end{rem}

\section{Conclusions}\label{fine}
In this paper, computational aspects related to the efficient
implementation of {\em HBVM methods with $k$ steps and degree $s$
($k\ge s$) (in short, HBVM$(k,s)$)},  have been recast in the
framework of {\em blended implicit methods}. In more details, we
have seen that the discrete problem generated by HBVM$(k,s)$
amounts to a nonlinear system of (block) dimension $s$. Its
efficient solution can be obtained by considering the {\em blended
formulation} of the discrete problem, for which an efficient {\em
diagonal splitting} can be easily defined. Consequently, to
implement the nonlinear iteration, only {\em one} matrix having
the same size as that of the continuous problem has to be
factored. The free parameter, on which the {\em blended iteration}
depends on, can be easily chosen, because of the {\em isospectral
property} of HBVM$(k,s)$, resulting in an {\em $L$-convergent
iteration}. Last, but not least, also the conditioning of the
discrete problem depends only on $s$, and it appears to tend to a
(nicely) bounded limit, as $k$ grows. We plan, in the future, to
implement HBVMs in blended formulation (in short, {\em blended
HBVMs}), in a computational code for numerically solving
Hamiltonian problems.


\begin{thebibliography}{99}
\setlength{\itemsep}{0cm}

\bibitem{AS} M.\,Abramovitz, I.A.\,Stegun. {\em Handbook of Mathematical
Functions}. Dover, 1965.

\bibitem{B00} L.\,Brugnano. Blended block BVMs (B$_3$VMs): A family of
economical implicit methods for ODEs. {\em J. Comput. Appl.Math.}
{\bf 116} (2000) 41--62.


\bibitem{BIS} L.\,Brugnano, F.\,Iavernaro, T.\,Susca. Hamiltonian BVMs (HBVMs):
implementation details and applications. ``Proceedings of ICNAAM
2009'', {\em AIP Conf. Proc.} {\bf 1168} (2009) 723--726.


\bibitem{BIT2} L.\,Brugnano, F.\,Iavernaro, D.\,Trigiante. Hamiltonian BVMs
(HBVMs): a family of ``drift-free'' methods for integrating polynomial
Hamiltonian systems. ``Proceedings of ICNAAM 2009'', {\em AIP
Conf. Proc.} {\bf 1168} (2009) 715--718.

\bibitem{BIT} L.\,Brugnano, F.\,Iavernaro, D.\,Trigiante. Analisys of Hamiltonian
Boundary Value Methods (HBVMs): a class of energy-preserving
Runge-Kutta methods for the numerical solution of polynomial
Hamiltonian dynamical systems. {\em BIT} (2009), submitted.
(arXiv:0909.5659)

\bibitem{BIT1} L.\,Brugnano, F.\,Iavernaro, D.\,Trigiante.
Hamiltonian Boundary Value Methods (Energy Preserving Discrete
Line Integral Methods). {\em Jour. of Numer. Anal., Industr. and
Appl. Math.} (2009) submitted. (arXiv:0910.3621)

\bibitem{BM02} L.\,Brugnano, C.\,Magherini. Blended implementation of block
implicit methods for ODEs. {\em Appl. Numer. Math.} {\bf 42}
(2002) 29--45.

\bibitem{BM04} L.\,Brugnano, C.\,Magherini. The {\tt BiM} code for the numerical
solution of ODEs. {\em J. Comput. Appl. Math.} {\bf 164--165}
(2004) 145--158.

\bibitem{BM07} L.\,Brugnano, C.\,Magherini. Blended implicit methods for solving
ODE and DAE problems, and their extension for second order
problems. {\em J. Comput. Appl. Math.} {\bf 205} (2007) 777--790.

\bibitem{BM08} L.\,Brugnano, C.\,Magherini. Blended General Linear Methods based
on Generalized BDF. {\em AIP Conf. Proc.} {\bf 1048} (2008)
871--874.

\bibitem{BM09} L.\,Brugnano, C.\,Magherini. Recent Advances in Linear Analysis
of Convergence for Splittings for Solving ODE problems. {\em Appl.
Numer. Math.} {\bf 59} (2009) 542--557.

\bibitem{BM09a} L.\,Brugnano, C.\,Magherini. Blended General Linear Methods
based on Boundary Value Methods in the GBDF family. {\em Journal
of Numerical Analysis, Industrial and Applied Mathematics} {\bf
4}, 1-2 (2009) 23--40.

\bibitem{BMM06} L\, Brugnano, C.\,Magherini, F.\,Mugnai. Blended implicit
methods for the numerical solution of DAE problems. {\em J.
Comput. Appl. Math.} {\bf 189} (2006) 34--50.

\bibitem{BT} L.\,Brugnano, D.\,Trigiante. {\em Solving
Differential Problems by Multistep Initial and Boundary Value
Methods}. Gordon and Breach, Amsterdam, 1998.

\bibitem{BT2} L.\,Brugnano, D.\,Trigiante. Block implicit methods for ODEs, in:
D. Trigiante (Ed.), {\em Recent Trends in Numerical Analysis}.
Nova Science Publ. Inc., New York, 2001, pp. 81--105.

\bibitem{HLW} E.\,Hairer, C.\,Lubich, G.\,Wanner. {\em Geometric
Numerical Integration. Structure-Preserving Algorithms for
Ordinary Differential Equations, 2$^{nd}$ ed.}, Springer, Berlin,
2006.

\bibitem{HW} E.\,Hairer, G.\,Wanner. {\em Solving Ordinary
Differential Equations II, 2$^{nd}$ ed.}, Springer, Berlin, 1996.

\bibitem{IP1} F.\,Iavernaro, B.\,Pace. $s$-Stage Trapezoidal Methods for the
Conservation of Hamiltonian Functions of Polynomial Type. {\em AIP
Conf. Proc.} {\bf 936} (2007) 603--606.

\bibitem{IP2} F.\,Iavernaro, B.\,Pace. Conservative Block-Boundary
Value Methods for the Solution of Polynomial Hamiltonian Systems.
{\em AIP Conf. Proc.} {\bf 1048} (2008) 888--891.

\bibitem{IT1} F.\,Iavernaro, D.\,Trigiante. Discrete conservative vector
fields induced by the trapezoidal method. {\em J. Numer. Anal.
Ind. Appl. Math.} {\bf 1} (2006) 113--130.

\bibitem{IT2} F.\,Iavernaro, D.\,Trigiante. State-dependent symplecticity
and area preserving numerical methods. {\em J. Comput. Appl.
Math.} {\bf 205} no.\,2  (2007) 814--825.

\bibitem{IT3} F.\,Iavernaro, D.\,Trigiante. High-order symmetric schemes for
the energy conservation of polynomial Hamiltonian problems. {\em
J. Numer. Anal. Ind. Appl. Math.} {\bf 4},1-2 (2009) 87--101.

\bibitem{M04} C.\,Magherini. {\em Numerical Solution of Stiff ODE-IVPs via
Blended Implicit Methods: Theory and Numerics.} PhD thesis,
Dipartimento di Matematica ``U.\,Dini'', Universit\`a degli Studi
di Firenze, September 2004 (Available at the url \cite{BIMcode}).

\bibitem{BIMcode} Codes {\tt BiM}/{\tt BiMD} Homepage:~ {\tt
http://www.math.unifi.it/\~{}brugnano/BiM/index.html}

\bibitem{TestSet} Test Set for IVP Solvers (release 2.4):~ {\tt
http://pitagora.dm.uniba.it/\~{}testset/}

\end{thebibliography}
\end{document}